\documentclass{amsart}
\usepackage{xspace,amssymb,bbm}
\theoremstyle{plain}
\newtheorem{nthr}{Theorem}[section]
\newtheorem{theorem}		[nthr]{Theorem}
\newtheorem{proposition}[nthr]{Proposition}
\newtheorem{lemma}		[nthr]{Lemma}
\newtheorem{corollary}	[nthr]{Corollary}

\newtheorem{conjecture} [nthr]{Conjecture}
\newtheorem{remark}		[nthr]{Remark}
\theoremstyle{definition}
\newtheorem{definition} [nthr]{Definition}
\newtheorem{example}		[nthr]{Example}

\def\mat#1{\ensuremath{#1}\xspace}
\def\dmat#1#2{\gdef#1{\mat{#2}}}
\def\oper#1#2{\dmat#1{\operatorname{#2}}}

\dmat\bk{\Bbbk}
\dmat\bA{\mathbb{A}}
\dmat\bC{\mathbb{C}}
\dmat\bF{\mathbb{F}}
\dmat\bG{\mathbb{G}}
\dmat\bH{\mathbb{H}}
\dmat\bL{\mathbb{L}}
\dmat\bN{\mathbb{N}}
\dmat\bQ{\mathbb{Q}}
\dmat\bR{\mathbb{R}}
\dmat\bT{\mathbb{T}}
\dmat\bZ{\mathbb{Z}}
\dmat\cC{\mathcal{C}}
\dmat\cG{\mathcal{G}}
\dmat\cP{\mathcal{P}}
\dmat\cT{\mathcal{T}}
\dmat\cV{\mathcal{V}}

\dmat\al\alpha
\dmat\be\beta
\dmat\ga\gamma
\dmat\de\delta
\dmat\eps{\varepsilon}
\dmat\hi\chi
\dmat\vi\varphi
\dmat\te\theta
\dmat\ka\kappa
\dmat\si\sigma
\dmat\Ga\Gamma
\dmat\Om\Omega

\def\ang#1{\mat{\left\langle #1\right\rangle}}

\def\set#1{\mat{\{#1\}}}
\def\sets#1#2{\mat{\{#1\mid#2\}}}
\def\sb{\subset}
\def\ms{\backslash}
\def\ts{\otimes}
\dmat\oh{\frac12}
\dmat\Gm{{\bG_{\mathrm m}}}

\def\inv{^{-1}}
\def\vir{{\mathrm{vir}}}
\def\n#1{\mat{\lvert#1\rvert}}
\def\nn#1{\mat{\lVert#1\rVert}}
\def\ie{i.e.\ }
\def\mto{\mapsto}
\def\lb#1{\mat{\underline{#1}}} 
\def\ub#1{\mat{\overline{#1}}}  
\def\udim{\lb\dim}
\def\bop{\bigoplus}
\def\pser#1{[\![#1]\!]} 
\def\xx{\times}
\def\iso{\simeq}
\def\over#1#2{\mat{\substack{#1\\#2}}} 
\dmat\es{\varnothing}
\def\arr^#1{\xrightarrow{#1}}
\def\ind{\mathrm{ind}}
\def\df{^\mathrm{def}}

\oper\Var{Var}
\oper\Hom{Hom}
\oper\GL{GL}
\oper\supp{supp}
\oper\End{End}
\oper\Iso{Iso}

\newif\ifrem\remtrue

\usepackage{hyperref}

\def\ob{\bar\Om^+}
\def\oa{\bar\Om^-}
\def\gm{\mathbf g}
\def\preceq{\le}
\def\prec{<}
\def\hq{Q}
\def\gme{q^{\oh}-q^{-\oh}} 
\def\GIT{/\!\!/}
\def\triv{\mathrm{triv}}
\def\nnn#1{[#1]}
\def\one{\mat{\mathbbm1}}
 
\remfalse

\begin{document}
\title[Abelian quiver invariants]{Abelian quiver invariants and marginal wall-crossing}
\author{Sergey Mozgovoy}%
\author{Markus Reineke}%
\email{mozgovoy@maths.tcd.ie}
\email{reineke@math.uni-wuppertal.de}
\date{\today}
\begin{abstract}  We prove the equivalence of (a slightly modified version of) the wall-crossing formula of Manschot, Pioline and Sen and the wall-crossing formula of Kontsevich and Soibelman. The former involves abelian analogues of the motivic Donaldson-Thomas type invariants of quivers with stability introduced by Kontsevich and Soibelman, for which we derive positivity and geometricity properties.
\end{abstract}

\maketitle

\section{Introduction}
Recently Manschot, Pioline, and Sen \cite{manschot_wall} proposed an explicit wall-crossing formula (called MPS wall-crossing in the following) for the BPS invariants in the rank two lattice of charges by using multi-centered black hole solutions in supergravity. They also conjectured that their formula is equivalent to the Kontsevich-Soibelman (KS) wall-crossing formula \cite{kontsevich_stability} in the refined case and to the Joyce-Song wall-crossing formula \cite{joyce_theory} in the unrefined case. In this paper we will show that a slightly modified MPS formula is equivalent to the KS wall-crossing formula.

A key ingredient of the MPS wall-crossing formula are abelian analogues of the motivic Donaldson-Thomas invariants of \cite{kontsevich_cohomological}. We will prove integrality and positivity properties of these and related invariants, and confirm a hypothesis of \cite{mps_blackhole} on their geometric nature.\\[1ex]
Let us describe first the KS wall-crossing formula (or HN recursion).
Let $\Ga$ be a rank $2$ lattice with a non-degenerate, integer valued skew-symmetric form $\ang{-,-}$ and let $\Ga_+\sb\Ga$ be a monoid having two generators. We define a total preorder on $\Ga_+^*=\Ga_+\ms\set0$ by setting $\al\le\be$ if $\ang{\al,\be}\ge0$. Similarly we order rays $l=\bR_{>0}\ga\sb\Ga\ts\bR$ with $\ga\in\Ga_+^*$.
Assume now that we have two families of (refined, rational DT) invariants $\oa_\ga,\ob_\ga$ for $\ga\in\Ga_+^*$, which are related by an equation of ordered products over rays taken in clockwise (resp.~anticlockwise) order
\begin{equation}
\prod^{\curvearrowright}_l\exp\left(\frac{\sum_{\ga\in l\cap\Ga}\ob_\ga x^\ga}{\gme}\right)
=\prod^{\curvearrowleft}_l\exp\left(\frac{\sum_{\ga\in l\cap\Ga}\oa_\ga x^\ga}{\gme}\right)
\label{eq:ks1}
\end{equation}
in the quantum torus (see Section \ref{subs:qa}) of $\Ga$. This is the KS wall-crossing formula. It allows us to recursively express the invariants $\ob_\ga$ in terms of the invariants $\oa_\ga$. For any nonzero $m:\Ga_+^*\to\bN$ with finite support define $\nn m=\sum m(\al)\al\in\Ga_+^*$ and $m!=\prod m(\al)!\in\bN$. Then we can write (cf.\ \cite[Eq.1.5]{manschot_wall})
\begin{equation}
\ob_\ga=\sum_{\over{m:\Ga_+^*\to\bN}{\nn m=\ga}}\frac{g(m)}{m!}\prod_{\al\in\Ga_+^*}(\oa_\al)^{m(\al)}
\label{eq:mps1}
\end{equation}
for some invariants $g(m)$. The computation of these invariants is recursive and is rather difficult (see however \cite{reineke_harder-narasimhan}). 

Manschot, Pioline, and Sen \cite{manschot_wall} suggested the following description of the invariants $g(m)$. They first construct invariants $g(\al_1,\dots,\al_n)$ for non-parallel $\al_i\in\Ga$, then extend their formula to non-parallel $\al_i\in\Ga\ts\bR$ and finally take limits to allow identical or parallel $\al_i$. The map $m:\Ga_+^*:\to\bN$ corresponding to $(\al_1,\dots,\al_n)$ is given by $m(\al)=\#{\set{i:\al_i=\al}}$. The equivalence of the KS wall-crossing formula \eqref{eq:ks1} and the MPS wall-crossing formula \eqref{eq:mps1} (with the above description of the invariants $g(m)$) was proved by Sen \cite{sen_equivalence}.

In this paper we give a slightly different description of the invariants $g(m)$ which leads us to the introduction of the abelian quiver invariants. For any $m:\Ga_+^*\to\bN$ define a quiver $\hq(m)$ with vertices $\al_k$, where $\al\in\Ga_+^*$ and $1\le k\le m(\al)$. Let the number of vertices from $\al_k$ to $\be_{k'}$ be $\ang{\be,\al}$ if $\ang{\be,\al}>0$ and zero otherwise. We define the invariant $f_+(m)$ to be the motivic invariant of the moduli stack of semistable abelian representations of $\hq(m)$, where abelian means that the representation has dimension one at every vertex of the quiver.
An explicit formula for the invariant $f_+(m)$ can be obtained by using the results of \cite{reineke_harder-narasimhan} (see also Corollary \ref{cr:abelian solution}).
Given a ray $l\sb\Ga\ts\bR$, we define invariants $g(m)$ with $\nn m\in l$, by the formula
\begin{equation}
1+\sum_{\nn m\in l}f_+(m)\frac{x^m}{m!}
=\exp\Bigg(\frac{\sum_{\nn m\in l}g(m)\frac{x^m}{m!}}{\gme}\Bigg).
\label{eq:mps2}
\end{equation}
Note that if $\nn m\in \Ga_+^*$ is indivisible then $g(m)=(\gme)f_+(m)$.
Our first result is the following equivalence conjectured by Manschot, Pioline, and Sen \cite{manschot_wall}.
\begin{theorem}
The MPS wall-crossing formula (equations \eqref{eq:mps1} and \eqref{eq:mps2}) is equivalent to the KS wall-crossing formula (equation \eqref{eq:ks1}).
\end{theorem}

For the proof of the above theorem we will closely study abelian quiver representations. Let $Q$ be a quiver with a fixed stability function $Z$ on the group of dimension vectors $\Ga(Q)$. For any dimension vector $\al\in\Ga(Q)$ we construct a new quiver $Q(\al)$ with vertices $i_k$, where $i\in Q_0$ and $1\le k\le\al_i$. Let the number of arrows from $i_k$ to $j_{k'}$ in $Q(\al)$ be the number of arrows from $i$ to $j$ in $Q$ (compare this construction with the above construction of the quiver $\hq(m)$). The quiver $Q(\al)$ inherits a stability function from $Q$, and we can define the moduli space of abelian semistable representations of $Q(\al)$. Let $f_Z(\al)$ be the motivic invariant of this moduli space. Our second result is the following analog of the HN recursion \cite{reineke_harder-narasimhan} (or KS wall-crossing formula \cite{kontsevich_stability})
\begin{theorem}
The ordered product
\begin{equation}
\prod_{l}^\curvearrowright\bigg(1+\sum_{Z(\al)\in l}f_Z(\al)\frac{x^\al}{\al!}\bigg)
\end{equation}
in the quantum torus of $Q$ is independent of the stability function $Z$.
\end{theorem}

The recursion formula that one obtains from the above theorem can be solved using the method of \cite{reineke_harder-narasimhan} (see Corollary \ref{cr:abelian solution}). We define abelian quiver invariants $g_Z(\al)$ by the formula
\begin{equation}
1+\sum_{Z(\al)\in l}f_Z(\al)\frac{x^\al}{\al!}
=\exp\Bigg(\frac{\sum_{Z(\al)\in l}g_Z(\al)\frac{x^\al}{\al!}}{\gme}\Bigg).
\end{equation}
Similarly to the case of motivic DT invariants of quivers with stability \cite{kontsevich_cohomological}, we can ask about polynomiality, integrality, and positivity properties of the invariants $g_Z(\al)$. Our third result is

\begin{theorem} Assume that the ray $l$ is such that $\langle \alpha,\beta\rangle=0$ whenever $Z(\alpha),Z(\beta)\in l$ and $f_Z(\alpha),f_Z(\beta)\not=0$. Then $g_Z(\al)\in\bN[q^{\pm\oh}]$.
\end{theorem}

In the case of a trivial stability $Z$, we give an explicit formula for $g_Z(\al)$ in terms of natural statistics on spanning trees in a graph. In the general case, the above theorem follows from

\begin{theorem} Under the assumptions of the previous theorem, the abelian invariant $g_Z(\al)$ equals the motive $[M_{Z\df}(Q(\alpha))]_{\rm vir}$ of the moduli space of abelian representations which are stable with respect to a suitably deformed stability $Z\df$.
\end{theorem}

This result confirms a "geometricity" hypothesis for motivic DT invariants implicit in \cite{mps_blackhole}; namely, there it is assumed that the motivic DT invariant for a quiver with stability always equals the motive of some appropriately defined "moduli space" depending on the quiver, the stability function and the dimension vector.\\[1ex]
The paper is organized as follows: in Section \ref{prelim} we recall basic facts on quivers, stability functions and the associated moduli spaces. Using Harder-Narasimhan techniques, we prove the abelian wall-crossing formula in Section \ref{sec:abelian}. We apply a graph-theoretic lemma in Section \ref{sec:gesselwang} to obtain an explicit formula for the abelian invariants for trivial stabilities. This enables us to discuss the notion of (quantum) admissibility of certain series in quantum tori, in analogy to \cite{kontsevich_cohomological}, in Section \ref{sec:admissible}, and to prove integrality of the abelian invariants. Section \ref{geometricity} contains the proof of the geometricity hypothesis for the abelian invariants, leading to their positivity properties. Motivated by this, we study related invariants counting indecomposable semistable abelian quiver representations in Section \ref{sec:indec}, and discuss their positivity properties and a graph-theoretic interpretation. The equivalence of KS and MPS wall-crossing is proved in Section \ref{mpswcf}; finally, we use abelian wall-crossing in Section \ref{mpsformula} to give a conceptual explanation for the motivic MPS degeneration formula of \cite{reineke_mps}.\\[1ex]
{\bf Acknowledgments:}
The first named author would like to thank Tam\'as Hausel, 
Boris Pioline and Ashoke Sen for helpful discussions.
 The second named author would like to thank Jan Manschot for explaining \cite{mps_blackhole}, and Sven Meinhardt, Jacopo Stoppa and Thorsten Weist for helpful discussions. 

\section{Notation and preliminaries}
\label{prelim}

\subsection{Motivic invariants}
In this section we assume that $\bk=\bC$. Let $K_0(\Var_\bk)$ be the group generated by isomorphism classes $[X]$ of algebraic varieties $X$ over \bk, subject to the relation $[X]=[Y]+[X\ms Y]$ for any closed subvariety $Y\sb X$. Let $\bL=[\bA^1]$ and let $\cV=K_0(\Var_\bk)\ts_{\bZ[\bL]}\bQ(\bL^\oh)$.
Sometimes we will denote $\bL$ by $q$.
For any smooth connected algebraic variety $X$ we define
\begin{equation}
[X]_\vir=q^{-\oh\dim X}[X]\in\cV.
\end{equation}
In particular, we define
\begin{equation}
\gm=[\Gm]_\vir=q^{-\oh}(q-1)=\gme.
\end{equation}

\begin{remark}
Given a smooth projective variety $X$ over \bk, we define its Poincar\'e polynomial $P(X)\in\bZ[q^{\oh}]$ by
\begin{equation}
P(X)=\sum_{k\ge0}q^{\frac k2}\dim H^k(X,\bQ).
\end{equation}
This map can be uniquely extended to a map $P:\cV\to\bQ(q^\oh)$ with $P(\bL^\oh)=q^{\oh}$, called the virtual Poincar\'e polynomial. Note that for any smooth projective variety $X$ the function
$$P([X]_\vir)=\sum_{k\ge0}q^{\oh(k-\dim X)}\dim H^k(X,\bQ)$$
is invariant under the change of variables $q^\oh\mto q^{-\oh}$ by Poincar\'e duality. 
\end{remark}

\subsection{Partitions}
Given a commutative monoid $S$ with identity element $0\in S$, let $S^*=S\ms\set0$. Given a set $X$ and a commutative monoid $S$, let $\cP(X,S)$ be the set of functions $f:X\to S$ with finite support, \ie functions such that $f\inv(S^*)$ is finite. We will denote $\cP(X,\bN)$ by $\cP(X)$. 
Given $f\in\cP(X)$, we define $f!=\prod_{x\in X}f(x)!$.
Note that $\cP(\bN^*)$ can be identified with the set of partitions \cite[\S1]{macdonald_symmetric}, as we can associate with any $m\in\cP(\bN^*)$ the partition $(1^{m_1},2^{m_2},\dots)$ having weight $\sum_{i\ge1}im_i$.
We define maps
\begin{equation}
\nn-:\cP(S^*)\to S,\qquad m\mto\sum_{s\in S}m(s)s
\end{equation}
and
\begin{equation}
\n-:\cP(X,S)\to S,\qquad f\mto\sum_{x\in X}f(x).
\end{equation}
For any $f\in\cP(X,S)$ define the multiplicity function $m_f\in\cP(S^*)$ by
$$S^*\ni s\mto\#f\inv(s).$$
Then $\n f=\nn{m_f}$. 

\subsection{Quivers}
Let $Q$ be a quiver (possibly infinite). We define the group of dimension vectors $\Ga(Q)=\cP(Q_0,\bZ)$ to be the group of maps $Q_0\to\bZ$ with finite support. Let $\Ga_+(Q)=\cP(Q_0)\sb\Ga(Q)$ be the monoid of maps $Q_0\to\bN$ with finite support and let $\Ga_+^*(Q)=\Ga_+(Q)\ms\set0$. For any vertex $i\in Q_0$ we denote also by $i\in\Ga_+(Q)$ the corresponding dimension vector $Q_0\ni j\mto\de_{ij}\in\bN$.

Define a bilinear form $r=r_Q$ on $\Ga(Q)$ by
\begin{equation}
r(\al,\be)=\sum_{(a:i\to j)\in Q_1}\al_i\be_j,
\label{eq:r}
\end{equation}
\ie for any $i,j\in Q_0$ the value $r(i,j)$ is the number of arrows from $i$ to $j$. Define the Euler-Ringel bilinear form $\hi=\hi_Q$ on $\Ga(Q)$ by
\begin{equation}
\hi(\al,\be)=\sum_{i\in Q_0}\al_i\be_i-\sum_{(a:i\to j)\in Q_1}\al_i\be_j=\al\cdot\be-r(\al,\be).
\end{equation} 
Finally, define a skew-symmetric form $\ang{\cdot,\cdot}$ on $\Ga(Q)$ by
\begin{equation}
\ang{\al,\be}=\hi(\al,\be)-\hi(\be,\al)=r(\be,\al)-r(\al,\be).
\end{equation}

\subsection{Stability functions}
A central charge (or stability function) on a quiver $Q$ is a group homomorphism $Z:\Ga(Q)\to\bC$ such that
$Z(j)\in\bH_+$ for $j\in Q_0$, where
\begin{equation}
\bH_+=\sets{re^{i\pi\vi}}{r>0,0<\vi\le 1}.
\end{equation}
There exist group homomorphisms $d,r:\Ga(Q)\to\bR$ such that
$Z(\al)=-d(\al)+ir(\al)$ for $\al\in\Ga(Q)$. Define the slope function $\mu_Z:\Ga_+^*(Q)\to\bR\cup\set\infty$ by the rule
$$\mu_Z(\al)=\frac{d(\al)}{r(\al)},\qquad \al\in \Ga(Q).$$
Define a total preorder $\preceq_Z$ on $\Ga_+^*(Q)$ by the rule $\al\preceq\be$ if $\mu_Z(\al)\le\mu_Z(\be)$. 
We will write $\al\prec_Z\be$ if $\al\preceq_Z\be$ but $\be\not\preceq_Z\al$, \ie if $\mu_Z(\al)<\mu_Z(\be)$. We define an equivalence relation $\sim_Z$ on $\Ga_+^*(Q)$ by $\al\sim_Z\be$ if $\mu_Z(\al)=\mu_Z(\be)$. The stability function $Z$ is called trivial if $\al\sim_Z\be$ for any $\al,\be\in\Ga_+^*(Q)$.

\subsection{Semistable representations}
For a representation $M$ of a quiver $Q$ we define its dimension vector $\udim M\in\Ga_+(Q)$ by $Q_0\ni i\mto\dim M_i$.
A representation $M$ is called stable (resp.\ semistable) if for any proper nonzero subrepresentation $N\sb M$ we have $\udim N\prec_Z\udim M$ (resp.\ $\udim N\preceq_Z\udim M$). For any $\al\in\Ga_+(Q)$, let $R_Z(Q,\al)$ be the subset of $Z$-semi-stable points in the space $R(Q,\al)=\bop_{a:i\to j}\Hom(\bk^{\al_i},\bk^{\al_j})$ of representations of $Q$ having dimension vector \al. It is endowed with an action of the group $\GL_\al(\bk)=\prod_{i\in Q_0}\GL_{\al_i}(\bk)$. We define the moduli space $M_Z(Q,\al)$ as the GIT quotient $R_Z(Q,\al)\GIT\GL_\al(\bk)$.
This moduli space is smooth if $Z$ is \al-generic, \ie every semistable representation having dimension vector \al is stable. 

For any ray $l\sb\bH_+$, we define the generating function
\begin{equation}
A_{Z,l}=1+\sum_{Z(\al)\in l}\frac{[R_Z(Q,\al)]_\vir}{[\GL_\al]_\vir}x^\al.
\end{equation}

Define the quantum torus $\bT_Q$ of the quiver $Q$ to be the algebra $\cV\pser{x_i,i\in Q_0}$ with the twisted multiplication
\begin{equation}
x^\al\circ x^{\be}=q^{\oh\ang{\al,\be}}x^{\al+\be}.
\end{equation}

The Harder-Narasimhan recursion formula \cite{reineke_harder-narasimhan} says that
\begin{equation}
\prod^{\curvearrowright}_l A_{Z,l}=\sum_{\al\in\Ga_+(Q)}\frac{[R(Q,\al)]_\vir}{[\GL_\al]_\vir}x^\al,
\label{eq:HN}
\end{equation}
where the product is taken over the rays $l\sb\bH_+$ in clockwise order.

\subsection{Purity and circle compact actions}
Following \cite{behrend_motivic} we say that an action of $\Gm$ on a variety $X$ is circle compact if the fixed point set $X^{\Gm}$ is proper and for any $x\in X$ the limit $\lim_{t\to 0}t\cdot x$ exists.

\begin{proposition}
Let $X,Y$ be varieties with an action of \Gm and let $f:X\to Y$ be a proper \Gm-equivariant morphism. If the action of \Gm on $Y$ is circle compact then the action of \Gm on $X$ is circle compact.
\end{proposition}
\begin{proof}
The variety $f\inv(Y^\Gm)$ is proper as $Y^\Gm$ is proper. 
The variety $X^{\Gm}$ is closed in $f\inv(Y^\Gm)$ and therefore is proper. For $x\in X$, by assumption, the map $\Gm\to Y$, $t\mto t\cdot f(x)$ can be extended to $\bA^1\to Y$. By the valuation criterion of properness of $f:X\to Y$, we can lift $\bA^1\to Y$ to a map $\bA^1\to X$ extending the map $\Gm\to X$, $t\mto t\cdot x$.
\end{proof}

\begin{proposition}[c.f.\ {\cite[Prop.~A.2]{crawley-boevey_absolutely}}]
Let $X$ be a smooth quasi-projective variety with a circle compact action of \Gm. Then the mixed Hodge structure on the cohomology of $X$ is pure and the virtual Poincar\'e polynomial of $X$ equals $$P(X,q^\oh)=\sum_{n\ge0}q^{\frac n2}\dim H^n_c(X,\bQ).$$
In particular, if $X$ is polynomial-count (see \cite[Section 6]{hausel_mixed}) with counting polynomial $P_X\in\bZ[q]$, then the odd cohomologies of $X$ vanish and $P_X\in\bN[q]$.
\end{proposition}
\begin{proof}
The variety $X^{\Gm}$ is smooth and projective.
Let $X^\Gm=\bigcup_iF_i$ be a decomposition into connected components. Consider the Bialynicki-Birula decomposition $X=\bigcup_i X_i$, where $X_i=\sets{x\in X}{\lim_{t\to0}t\cdot x\in F_i}$. 
This decomposition is filtrable (as $X$ is quasi-projective, c.f.\ \cite[Prop.~A.2]{crawley-boevey_absolutely}).
The natural projection $X_i\to F_i$ is an affine fibration. Every $F_i$ has pure Hodge structure, therefore the same is true for $X_i$ and therefore also for $X$.

If $X$ is polynomial-count with counting polynomial $P_X$ then $P(X,q^\oh)=P_X(q)\in\bZ[q]$ by \cite{hausel_mixed}.
This implies that $H^n(X,\bQ)=0$ for odd $n$ and $P_X\in\bN[q]$.
\end{proof}

\begin{proposition}[c.f.\ {\cite[Section 2.2]{engel_smooth}}]
\label{purity}
Let $Q$ be a quiver with a stability function $Z$. Let $\al\in\Ga_+(Q)$ be a dimension vector such that $Z$ is \al-generic (\ie any $Z$-semistable $Q$-representation of dimension vector \al is stable). Then the moduli space $M_Z(Q,\al)$ has a circle compact action of \Gm, it is polynomial-count with counting polynomial $P(q)$ in $\bN[q]$, and its motive equals $P(\bL)$.
\end{proposition}
\begin{proof}
Let $M_0(Q,\al)=R(Q,\al)\GIT\GL_\al$. Then the inclusion $R_Z(Q,\al)\to R(Q,\al)$ induces a projective map $\pi:M_Z(Q,\al)\to M_0(Q,\al)$.
Consider the action of $\Gm$ on $R(Q,\al)$ given by
$$t\cdot (M_a)_{a\in Q_1}=(tM_a)_{a\in Q_1}.$$
It induces an action of \Gm on $M_Z(Q,\al)$ and $M_0(Q,\al)$ such that $\pi$ is \Gm-equivariant. The action of \Gm on $M_0(Q,\al)$ is circle compact (the set of \Gm-invariant points consists of the zero representation). By the previous proposition the action of \Gm on $M_Z(Q,\al)$ is circle compact. We know that $M_Z(Q,\al)$ is polynomial-count. Therefore, by the previous proposition its counting polynomial has non-negative coefficients. Moreover, it follows from the motivic nature of the Harder-Narasimhan relation (discussed e.g. in \cite[Section 3.2]{reineke_mps}) that if $Z$ is \al-generic then the motive of $M_Z(Q,\al)$ equals $P(\bL)$.
\end{proof}

\section{Abelian wall-crossing formula}
\label{sec:abelian}
We say that a representation $M$ of the quiver $Q$ is abelian (or thin sincere)
if $\dim M_i=1$ for any $i\in Q$. Let $\one=\one_{Q_0}\in\Ga(Q)$ be the corresponding dimension vector, with $\one_i=1$ for $i\in Q_0$. Note that $\hi_Q(\one,\one)=\n{Q_0}-\n{Q_1}$.
The space of abelian representations $R(Q)=R(Q,\one)=\bk^{Q_1}$ is endowed with an action of the group $G(Q)=\GL_{\one}=(\bk^*)^{Q_0}$.
Given a stability function $Z:\Ga(Q)\to\bC$, let $R_Z(Q)\sb R(Q)$ be the subspace of abelian $Z$-semi-stable representations. We define
\begin{equation}
f_Z(Q)=\frac{[R_Z(Q)]_\vir}{[G(Q)]_\vir}
=q^{\oh(\n{Q_0}-\n{Q_1})}\frac{[R_Z(Q)]}{(q-1)^{\n{Q_0}}}.
\end{equation}

\begin{remark}
For a trivial stability all representations are semistable. Therefore
$$f_{\triv}(Q)
=q^{\oh(\n{Q_0}-\n{Q_1})}\frac{q^{\n{Q_1}}}{(q-1)^{\n{Q_0}}}
=\frac{q^{\oh\n{Q_1}}}{(\gme)^{\n{Q_0}}}.$$
\end{remark}

Given a dimension vector $\al\in\Ga_+(Q)$, we define a new quiver $Q(\al)$ with vertices $i_k$, where $i\in Q_0$ and $1\le k\le \al_i$. The number of arrows from $i_k$ to $j_{k'}$ is defined to be the number of arrows from $i$ to $j$. Then
\begin{equation}
\n{Q(\al)_0}=\n\al,\qquad \n{Q(\al)_1}=\sum_{(a:i\to j)\in Q_1}\al_i\al_j=r(\al,\al).
\end{equation}
Any stability function $Z:\Ga(Q)\to\bC$ on $Q$ induces a stability function on $Q(\al)$ by $Z(i_k)=Z(i)$. We define $f_Z(\al)=f_Z(Q(\al))$.
In particular, for a trivial stability, we have
\begin{equation}
f_{\triv}(\al)=\frac{q^{\oh r(\al,\al)}}{(\gme)^{\n\al}}.
\label{eq:trivial stab}
\end{equation}

\begin{remark}
Consider the group homomorphism
$$\pi:\Ga(Q(\al))\to\Ga(Q),\qquad i_k\mto i.$$
It maps the dimension vector $\one_{Q(\al)}\in\Ga(Q(\al))$ to $\al\in\Ga(Q)$.
The map $\pi$ preserves the skew-symmetric form. This implies that it induces an algebra homomorphism of the corresponding quantum tori.
For any $I\sb Q(\al)$ we define
$$\pi(I)=\pi(\one_I)=\sum_{i_k\in I}i\in\Ga(Q).$$

\end{remark}

\begin{theorem}
\label{th:wall-cross}
We have
\begin{equation}
\prod_{l}^\curvearrowright\bigg(1+\sum_{Z(\al)\in l}f_Z(\al)\frac{x^\al}{\al!}\bigg)
=\sum_{\al\in\Ga_+(Q)}f_{\triv}(\al)\frac{x^\al}{\al!}.
\label{eq:wall-cross}
\end{equation}
In particular, the product on the left is independent of the stability function $Z$.
\end{theorem}
\begin{proof}
Given an abelian representation $M$ of $Q(\al)$, there is a unique filtration (the Harder-Narasimhan filtration)
$$0=M_0\sb M_1\sb\dots\sb M_s=M$$
with semistable quotients having decreasing slopes. Each subquotient $M_k/M_{k-1}$ of $M$ is uniquely determined by its support $I_k$, the set of vertices of $Q(\al)_0$ where it is nonzero. We obtain a disjoint decomposition $Q(\al)_0=I_1\dot\cup\dots\dot\cup I_s$, called the HN type of $M$.

Given a partition $\al=\al^1+\dots+\al^s$, the number of ways to decompose $Q(\al)_0=I_1\dot\cup\dots\dot\cup I_s$ so that $\pi(I_k)=\al^k$ equals
$\binom{\al}{\al^1,\dots,\al^s}=\frac{\al!}{\al^1!\dots\al^s!}$.
This implies that, for a fixed $\al\in\Ga_+(Q)$, the expression
$$\sum_{\over{\al^1+\dots+\al^s=\al}{\al^1>_Z\dots>_Z\al^s}}
f_Z(\al^1)x^{\al^1}\circ\dots\circ f_Z(\al^s)x^{\al^s}\frac{\al!}{\al^1!\dots\al^s!}
$$
equals the invariant of the moduli stack of all abelian representations of $Q(\al)$. This proves the theorem.
\end{proof}

\begin{corollary}
\label{cr:abelian solution}
For any stability function $Z$ we have
$$f_Z(\al)=(\gme)^{-\n\al}
\sum_{\over{\al^1+\dots+\al^s=\al}{\al^1+\dots+\al^i>_Z\al}}
(-1)^{n-1}
\binom\al{\al^1,\dots,\al^s}
(q^\oh)^{\sum_{i<j}\ang{\al^i,\al^j}+\sum_i r(\al^i,\al^i)}.$$
\end{corollary}
\begin{proof}
It follows from the theorem that
$$\frac{f_{\triv}(\al)}{\al!}x^\al
=\sum_{\over{\al^1+\dots+\al^s=\al}{\al^1>_Z\dots>_Z\al^s}}
\frac{f_Z(\al^1)}{\al^1!}x^{\al^1}\circ\dots\circ \frac{f_Z(\al^s)}{\al^s!}x^{\al^s}.$$
Applying the formula of \cite[Theorem 5.1]{reineke_harder-narasimhan} for the solutions of such recursions we get 
$$\frac{f_Z(\al)}{\al!}x^\al=\sum_{\over{\al^1+\dots+\al^s=\al}{\al^1+\dots+\al^i>_Z\al}}(-1)^{n-1}
\frac{f_{\triv}(\al^1)}{\al^1!}x^{\al^1}\circ\dots\circ \frac{f_{\triv}(\al^s)}{\al^s!}x^{\al^s}.$$
This, together with equation \eqref{eq:trivial stab}, implies the corollary.
\end{proof}

\begin{remark}
One can obtain this formula also by applying directly \cite{reineke_harder-narasimhan} to the invariants of moduli spaces of semi-stable representations of $Q(\al)$ having dimension vector $\one_{Q(\al)}$.
\end{remark}

It follows from the previous results that, for any $\al\in\Ga_+^*(Q)$, the invariants $f_Z(\al)$ are rational functions in the variable $q^\oh$. 
We are going to define certain polynomial invariants now.

\begin{definition}[Abelian quiver invariants]
\label{df:ab1}
Assume that a ray $l\in\bH_+$ is such that 
\begin{equation}
\text{if}\quad Z(\al),Z(\be)\in l\quad\text{ and }\quad f_Z(\al),f_Z(\be)\ne0\quad\text{ then }\quad\ang{\al,\be}=0.
\label{eq:admiss ray}
\end{equation}
Then we define abelian quiver invariants $g_Z(\al)$ for $\al\in Z\inv(l)\cap \Ga_+(Q)$  by the formula
\begin{equation}
1+\sum_{Z(\al)\in l}f_Z(\al)\frac{x^\al}{\al!}
=\exp\Bigg(\frac{\sum_{Z(\al)\in l}g_Z(\al)\frac{x^\al}{\al!}}{q^\oh-q^{-\oh}}\Bigg).
\label{eq:ab1}
\end{equation}
As one of our main results, we will prove in Section \ref{geometricity} that $g_Z(\al)\in\bN[q^{\pm\oh}]$.
\end{definition}

\section{Combinatorial methods}\label{sec:gesselwang}
We recall a key lemma by Gessel and Wang \cite{gessel_depth} relating connected graphs and trees.

Let $X$ be a finite set with a total ordering, and denote the minimal element by $x_0\in X$. A graph with set of nodes $X$ is encoded as a subset $G$ of the set $\binom{X}{2}$ of two-element subsets of $X$; then there is an edge between nodes $k$ and $l$  if and only if $\{k,l\}\in G$. The number of edges is denoted by $e(G)$. If $G$ is connected, then $e(G)\geq | X|-1$.

Suppose that $T$ is a tree on $X$, that is, a connected graph such that every proper subgraph is non-connected. Then $e(T)=|X|-1$. The tree $T$ induces a partial ordering on $X$ if we view $T$ as rooted in the node $x_0$ as follows: for every $x\in X$, there exists a unique path from $x_0$ to $x$, say $x_0-x_1-\ldots-x_k=x$. Then $y\trianglelefteq x$ if and only if the path from $x_0$ to $x$ passes through $y$, that is, if $y=x_i$ for some $0\leq i\leq k$. In particular, we denote by $p(x)=x_{k-1}$ the immediate predecessor of an element $x\in X\setminus\{x_0\}$.

Define the inversion set $I(T)$ of $T$ as the set of all pairs $(x,y)\in X^2$ such that $x\triangleleft y$, but $x>y$.

\begin{lemma}\label{le:gw} There is a natural bijection between connected graphs on $X$ and pairs $(T,J)$ consisting of a tree on $X$ and a subset $J$ of $I(T)$.
\end{lemma}
\begin{proof}
(see \cite{gessel_depth}). To a pair $(T,J)$ we associate the graph with edges being those of $T$, together with edges $\{p(j),k\}$ for $(j,k)\in J$. Conversely, given a connected graph $G$, we construct a tree $T$ by performing a depth-first search, that is, $T$ is constructed recursively as follows: we first define $x$ as $x_0$. In each step, if possible, we choose the maximal successor $x_{\max}$ of $x$ which is not already a node of $T$, add the edge $\{x,x_{\max}\}$ to $T$ and replace $x$ by $x_{\max}$; otherwise, we replace $x$ by $p(x)$. Finally, we define $J$ as the set of all $(j,k)\in I(T)$ such that the edge $\{p(j),k\}$ belongs to $G\setminus T$.
\end{proof}


Now we assume $X$ to be an $n$-coloured set, that is, we choose a function $c:X\rightarrow\{1,\ldots,n\}$. We want to enumerate graphs (resp.\ connected graphs, resp.\ trees) according to the colours of the nodes. We choose indeterminates $t_{ij}=t_{ji}$ for $1\leq i\leq j\leq n$ and define the weight $t^G$ of a graph $G$ on $X$ by $\prod_{\{x,y\}\in G}t_{c(x),c(y)}$.\\

For a tuple $\al=(\al_1,\ldots,\al_n)$ of nonnegative integers, we choose a finite $n$-coloured set $X$ containing $\al_i$ elements of colour $i$ for $i=1,\ldots,n$. Let $\mathcal{G}_\al$ be the set of graphs on $X$, let $\mathcal{C}_\al$ be the subset of connected graphs, and let $\mathcal{T}_\al$ be the subset of trees. We choose indeterminates $z_1,\ldots,z_n$ and denote $\frac{z^\al}{\al!}=\frac{z_1^{\al_1}}{\al_1!}\cdot\ldots\cdot \frac{z_n^{\al_n}}{\al_n!}$.\\

By the exponential formula \cite[Corollary 5.1.6]{stanley_enumerative2}, we have
\begin{equation}\label{zt}\sum_{\al\in\bN^n}\sum_{G\in\mathcal{G}_\al}t^G\frac{z^\al}{\al!}=\exp\bigg(\sum_{\al\in\bN^n\ms\set0}\sum_{G\in\mathcal{C}_\al}t^G\frac{z^\al}{\al!}\bigg).\end{equation}
The left hand side of equation \eqref{zt} can be made explicit, noting that the choice of a graph $G$ is just the choice of an arbitrary subset of $\binom{X}{2}$, namely
$$\sum_{\al\in\bN^n}\sum_{G\in\mathcal{G}_\al}t^G\frac{z^\al}{\al!}=\sum_{\al\in\bN^n}\prod_i(1+t_{ii})^{\binom{\al_i}{2}}\prod_{i<j}(1+t_{ij})^{\al_i\al_j}\frac{z^\al}{\al!}.$$

Using Lemma \ref{le:gw}, we can also rewrite the right hand side of equation \eqref{zt}, namely as
$$\exp\bigg(\sum_{\al\in\bN^n\ms\set0}\sum_{T\in\mathcal{T}_\al}\prod_{\{x,y\}\in T}t_{c(x),c(y)}\prod_{(x,y)\in I(T)}(1+t_{c(p(x)),c(y)})\frac{z^\al}{\al!}\bigg).$$

Putting these equations together, we arrive at

$$\sum_{\al\in\bN^n}\prod_i(1+t_{ii})^{\binom{\al_i}{2}}\prod_{i<j}(1+t_{ij})^{\al_i\al_j}\frac{z^\al}{\al!}=$$
$$=\exp\bigg(\sum_{\al\in\bN^n\ms\set0}\sum_{T\in\mathcal{T}_\al}\prod_{\{x,y\}\in T}t_{c(x),c(y)}\prod_{(x,y)\in I(T)}(1+t_{c(p(x)),c(y)})\frac{z^\al}{\al!}\bigg).$$

Let $r=(r_{ij})$ be a symmetric integer $n\xx n$ matrix.
We replace $t_{ij}$ by \linebreak
$q^{\oh(r_{ij}+r_{ji})}-1=q^{r_{ij}}-1$ for $1\leq i,j\leq n$ and replace $z_i$ by $\frac{x_i}{q-1}$. Using the fact that a tree in $\mathcal{T}_\al$ contains precisely $\sum_i\al_i-1$ edges, the previous equality can be rewritten as
\begin{multline*}
\sum_{\al\in\bN^n}\frac{q^{\oh\sum_{i,j}r_{ij}\al_i\al_j}}{(q-1)^{\sum_i\al_i}}\frac{x^\al}{\al!}\\
=\exp\bigg(\sum_{\al\in\bN^n\ms\set0}
\frac{q^{\oh\sum_i r_{ii}\al_i}}{q-1}\frac{x^\al}{\al!}
\sum_{T\in\mathcal{T}_\al}\prod_{\{x,y\}\in T}
\frac{q^{r_{c(x),c(y)}}-1}{q-1}
\prod_{(x,y)\in I(T)}q^{r_{c(p(x)),c(y)}}\bigg).
\end{multline*}

We have thus proved:

\begin{theorem}\label{pos}
Given a symmetric integer $n\xx n$ matrix $r=(r_{ij})$, there exist polynomials $b_\al\in\bZ[q^{\pm \oh}]$ for $\al\in\bN^n\ms\set0$ such that
$$\sum_{\al\in\bN^n}\frac{q^{\oh\sum_{i,j}r_{ij}\al_i\al_j}}{(q-1)^{\sum_i\al_i}}\frac{x^\al}{\al!}=\exp\bigg(\frac{1}{q-1}\sum_{\al\in\bN^n\ms\set0}b_\al\frac{x^\al}{\al!}\bigg).$$
If all $r_{ij}$ are nonnegative then $b_\al\in\bN[q^{\oh}]$ for $\al\in\bN^n\ms\set0$.
\end{theorem}

\begin{corollary}\label{corpos}
Under the assumptions of the theorem, the value of the polynomial $b_\al$ at $q^{\frac{1}{2}}=1$ equals
$$b_\al(1)=\sum_{T\in\mathcal{T}_\al}\prod_{\{x,y\}\in T}r_{c(x),c(y)}.$$
\end{corollary}

The above theorem in particular applies to the matrix $r=r_Q$ of a symmetric   quiver $Q$ (that is, $r(i,j)=r(j,i)$ for all $i,j\in Q_0$) and the trivial stability.
By Definition \ref{df:ab1} and formula \eqref{eq:trivial stab}
$$\sum_{\al\in\Ga_+(Q)}\frac{q^{\oh r(\al,\al)}}{(\gme)^{\n\al}}\frac{x^\al}{\al!}
=\exp\Bigg(\frac{\sum_{\al\in\Ga_+^*(Q)}g_\triv(\al)\frac{x^\al}{\al!}}{q^\oh-q^{-\oh}}\Bigg).
$$
Using notation of Theorem \ref{pos} we obtain
$$g_\triv(\al)=q^{\oh(\n\al-1)}b_\al\in\bN[q^{\oh}].$$
Application of Corollary \ref{corpos} yields
$$g_\triv(\al)|_{q^\oh=1}=\sum_{T\in\cT_\al}\prod_{\set{i_k,j_l}\in T}r(i,j),$$
where we identify $\cT_\al$ with the set of trees on $Q(\al)_0$.
Define an unoriented graph $\overline{Q}(\alpha)$ with the set of vertices $Q(\al)_0$, and with the number of edges between $i_k$ and $j_l$ being the number of arrows from $i$ to $j$ in $Q$ (this being well-defined since $Q$ is symmetric).

\begin{corollary}
\label{triv_special}
The value of $g_{\triv}(\alpha)$ at $q^\oh=1$ equals the number of spanning trees
of $\overline{Q}(\alpha)$.
\end{corollary}

\begin{example} The previous result allows us to easily derive explicit formulas for the value of $g_{\triv}(\alpha)$ at $q^{\frac{1}{2}}=1$ for small quivers $Q$:
\begin{enumerate}
\item For the quiver with a single vertex $i$ and $m$ loops, we get the formula $$g_{\triv}(d\cdot i)|_{q^{\frac{1}{2}}=1}=m^{d-1}d^{d-2}:$$
by Cayley's theorem, there are $d^{d-2}$ spanning trees in a complete graph with $d$ nodes. For each of the $d-1$ edges of this tree, we can freely choose a colour from $1$ to $m$.
\item For the quiver with two vertices $i$ and $j$ and $a$ arrows from $i$ to $j$ and from $j$ to $i$, we get the formula $$g_{\triv}(d_i\cdot i+d_j\cdot j)|_{q^{\frac{1}{2}}=1}=a^{d_i+d_j-1}d_i^{d_j-1}d_j^{d_i-1}:$$
using the matrix-tree theorem \cite[Theorem 5.6.8]{stanley_enumerative2}, the number of spanning trees in a complete bipartite quiver with $d_i$ nodes on one side and $d_j$ nodes on the other side can be computed to be $d_i^{d_j-1}d_j^{d_i-1}$; again, we are allowed to choose among $a$ colours for each of the $d_i+d_j-1$ edges of such a spanning tree.
\end{enumerate}
\end{example}
\section{Admissible series}\label{sec:admissible}

Given a set $I$, we define
$$\Ga(I)=\cP(I,\bZ),\qquad \Ga_+(I)=\cP(I,\bN),\qquad \Ga_+^*(I)=\Ga_+(I)\ms\set0.$$
Let $r:I\xx I\to\bZ$ be some symmetric function. We can naturally extend it to a symmetric bilinear form $r:\Ga(I)\xx\Ga(I)\to\bZ$. Let $\bT=\bQ(q^\oh)\pser{x_i,i\in I}$ and
\begin{equation}
T_r:\bT\to\bT,\qquad x^\al\mto q^{\oh r(\al,\al)}x^\al,\quad \al\in\Ga_+(I).
\label{eq:1}
\end{equation}
In this notation, Theorem \ref{pos} can be rephrased as
$$T_r\exp\bigg(\frac{\sum x_i}{q-1}\bigg)
=\exp\bigg(\frac{\sum b_\al x^\al/\al!}{q-1}\bigg)
$$
for some polynomials $b_\al\in\bZ[q^{\pm\oh}]$, $\al\in\Ga_+^*(I)$. Moreover, if $r(i,j)\ge0$ for $i,j\in I$ then $b_\al\in\bN[q^\oh]$.
Using this result we are going to prove

\begin{theorem}
\label{th:2}
Let $a_\al\in\bZ[q^{\pm\oh}]$ for $\al\in\Ga_+^*(I)$. Then
$$T_r\exp\Big(\frac{\sum a_\al x^\al/\al!}{q-1}\Big)
=\exp\Big(\frac{\sum b_\al x^\al/\al!}{q-1}\Big)$$
for some polynomials $b_\al\in\bZ[q^{\pm\oh}]$. Moreover, if $r(i,j)\ge0$ for $i,j\in I$ and $a_\al\in\bN[q^\oh]$ for $\al\in\Ga_+^*(I)$ then $b_\al\in\bN[q^\oh]$ for $\al\in\Ga_+^*(I)$.
\end{theorem}
\begin{proof}
For any $\al\in I'=\Ga_+^*(I)$ we define a new variable $y_\al$ and we substitute $a_\al x^\al/\al!$ by this variable in the above expression. The bilinear form $r:\Ga(I)\xx\Ga(I)\to\bZ$ restricts to a map $r:I'\xx I'\to\bZ$ and then extends to a bilinear form $r':\Ga(I')\xx\Ga(I')\to\bZ$. We can define an operator $T'=T_{r'}$ on the algebra $\bT'=\bQ(q^\oh)\pser{y_\al,\al\in I'}$ using this bilinear form.
Then, by the previous theorem,
$$T'\exp\Big(\frac{\sum y_\al}{q-1}\Big)
=\exp\Big(\frac{\sum_{\ga\in\Ga_+^*(I')} b_\ga(q)\prod _{\al\in I'}\frac{y_\al^{\ga(\al)}}{\ga(\al)!}}{q-1}\Big).
$$
Now we substitute $y_\al$ by $a_\al(q) x^\al/\al!$. To prove the theorem we have to show that
$$\frac{(\sum_\al\ga(\al)\al)!}{\prod_{\al}\ga(\al)!(\al!)^{\ga(\al)}}\in\bN$$
for every $\ga\in\Ga_+^*(I')$.
For any $\al\in\Ga(I)$ let $S_\al=\prod_{i\in I}S_{\al_i}$ be the product of permutation groups. Then $\prod_\al (S_\al)^{\ga(\al)}\rtimes S_{\ga(\al)}$ can be embedded into $S_{\sum_\al\ga(\al)\al}$. This proves the statement.
\end{proof}

\subsection{Quantum admissibility}\label{subs:qa}
We now repeat some results from \cite[\S6.1]{kontsevich_cohomological} in the context of abelian invariants. Let $I=\set{1,\dots,n}$ be a finite set and let $\ang{-,-}$ be a skew-symmetric bilinear form on $\Ga(I)$. We define a quantum torus structure on $\bT=\bQ(q^\oh)\pser{x_i,i\in I}$ by
$$x^\al\circ x^\be=q^{\oh\ang{\al,\be}}x^{\al+\be}.$$

\begin{definition}
We say that a series in \bT is admissible if it is of the form
$$\exp\Big(\frac{\sum b_\al x^\al/\al!}{q-1}\Big)$$
for some polynomials $b_\al\in\bZ[q^{\pm\oh}]$. We say that a series
$\sum a_\al x^\al\in\bT$ is quantum admissible if
$$\sum_\al a_\al x_1^{\al_1}\circ\dots\circ x_n^{\al_n}$$
is admissible.
\end{definition}

\begin{proposition}
\label{pr:1}
The set of quantum admissible series forms a group in $\bT$ under multiplication $\circ$. Let $Z:\Ga(I)\to\bC$ be a stability function,
$F,F_l\in\bT$ with $F_l=1+\sum_{Z(\al)\in l}a_\al x^\al$, and
$$F=\prod^{\curvearrowright}_{l\sb\bH_+}F_l.$$
Then $F$ is quantum admissible if and only if every $F_l$ is quantum admissible.
\end{proposition}
\begin{proof}
The proof for a generic stability function (this means that $\al,\be\in Z\inv(l)\cap\Ga_+(I)$ implies proportionality $\al\parallel\be$) goes through the lines of \cite[Prop.~9]{kontsevich_cohomological} (the main tool used there is \cite[Theorem 9]{kontsevich_cohomological} which is a plethystic analogue of Theorem \ref{th:2}). This proof works actually for an arbitrary stability function (using the fact that quantum admissible series form a group). The proof of the fact that quantum admissible series form a group goes through the lines of \cite[Prop.~10]{kontsevich_cohomological}.
\end{proof}

\begin{proposition}
Let $F=\sum a_\al x^\al$ be such that $\ang{\al,\be}=0$ whenever $a_\al,a_\be\ne0$. Then $F$ is quantum admissible if and only if it is admissible.
\end{proposition}
\begin{proof}
Let $Z$ be a generic stability and let $F=\prod^{\curvearrowright}_{l\sb\bH_+}F_l$.
One can show that if $a_\al$ is a nontrivial coefficient of $F_l$ and $a_\be$ is a nontrivial coefficient of $F_{l'}$ then $\ang{\al,\be}=0$. This implies that $F_{l}\circ F_{l'}=F_{l}F_{l'}$ and $F=\prod_lF_l$. Now we use the fact that for $F_l$ admissibility and quantum admissibility coincide.
\end{proof}

\begin{corollary}
\label{cr:1}
Let $Q$ be a quiver and $Z:\Ga(Q)\to\bC$ be a stability function.
Then for any ray $l\in\bH_+$ the series
$$1+\sum_{Z(\al)\in l}f_Z(\al)\frac{x^\al}{\al!}$$
is quantum admissible. If $\ang{\al,\be}=0$ whenever $Z(\al),Z(\be)\in l$ and $f_Z(\al),f_Z(\be)\ne0$ then the above expression is admissible.
\end{corollary}
\begin{proof}
It is enough to prove quantum admissibility for the trivial stability.
For the trivial stability we have
$$f_{Z}(\al)=\frac{(q^\oh)^{\al\cdot\al-\hi(\al,\al)}}{(q^\oh-q^{-\oh})^{\n\al}}.$$
We have to verify admissibility for these invariants shifted by $q^{\oh\sum_{i<j}\ang{i,j}\al_i\al_j}$.
Note that the series
$$\sum_{\al\in\Ga_+(Q)}\frac{1}{(\gme)^{\n\al}}\frac{x^\al}{\al!}=\exp\bigg(\frac{\sum x_i}{\gme}\bigg)$$
is admissible. In order to apply Theorem \ref{th:2} we have to show that the quadratic form
$$\al\mto\sum_{i<j}\ang{i,j}\al_i\al_j-\hi(\al,\al)$$
is given by a symmetric matrix with integer entries. But it equals
$$-\sum_{i<j}\hi(j,i)\al_i\al_j-\sum_{i\ge j}\hi(i,j)\al_i\al_j
=-\sum_i\hi(i,i)\al_i^2-\sum_{i<j}2\hi(j,i)\al_i\al_j.$$
\end{proof}

\begin{corollary}
Under the assumptions of Definition \ref{df:ab1}, the abelian quiver invariants $g_Z(\al)$ are elements of $\bZ[q^{\pm\oh}]$.
\end{corollary}
\begin{proof}
By the previous corollary the series
$$1+\sum_{Z(\al)\in l}f_Z(\al)\frac{x^\al}{\al!}$$
is admissible. Therefore, the invariants $g_Z(\al)$ defined by
$$1+\sum_{Z(\al)\in l}f_Z(\al)\frac{x^\al}{\al!}
=\exp\Bigg(\frac{\sum_{Z(\al)\in l}g_Z(\al)\frac{x^\al}{\al!}}{q^\oh-q^{-\oh}}\Bigg)$$
satisfy $g_Z(\al)\in\bZ[q^{\pm\oh}]$.
\end{proof}

\section{Geometricity of abelian invariants}\label{geometricity}

The aim of this section is to prove that the abelian invariant $g_Z(\al)$ (under the numerical conditions of Definition \ref{df:ab1}) can be naturally interpreted as the virtual motive $[M_{Z\df}(Q(\al))]_\vir$ of a moduli space $M_{Z\df}(Q(\al))=R_{Z\df}(Q(\al))/G(Q(\al))$ satisfying the properties of Proposition \ref{purity}. In particular, this implies the positivity of $g_Z(\al)$ as a polynomial in $\bL$. The strategy is to deform the stability $Z$ to a generic one $Z\df$, to analyze the Harder-Narasimhan stratification (with respect to $Z\df$) of the $Z$-semistable locus, and to show that the corresponding relative Harder-Narasimhan recursion equals the defining equation for $g_Z(\al)$.

For any subset $I\sb Q_0$ we denote $\mu_Z(\one_I)$ by $\mu_Z(I)$. For an abelian representation $M$ of $Q$, the subset $I$ is called $M$-closed if for any arrow $\alpha:i\rightarrow j$, we have $j\in I$ if $i\in I$ and $M_\alpha\not=0$. Note that the $M$-closed subsets are precisely the supports of subrepresentations of $M$.


\begin{lemma}
\label{lm:deform}
Given a quiver $Q$ with a stability $Z$, there exists a stability $Z\df$ on $Q$ such that
\begin{enumerate}
	\item $\mu_{Z\df}(I)\ne\mu_{Z\df}(J)$ for any non-empty subsets $I\ne J$ of $Q_0$.
	\item If $\mu_Z(I)<\mu_Z(J)$ then $\mu_{Z\df}(I)<\mu_{Z\df}(J)$.
\item An abelian representation $M\in R(Q)$ is $Z\df$-stable if and only if it is $Z\df$-semistable.
\item If an abelian representation $M\in R(Q)$ is $Z\df$-stable, it is $Z$-semistable.
\item If an abelian representation $M\in R(Q)$ is $Z$-stable, it is $Z\df$-stable.
\end{enumerate}
\end{lemma}
\begin{proof}
There exists a stability $Z':\bZ^{Q_0}\to\bC$ satisfying the first condition. Indeed, for any non-empty subsets $I\ne J$ of $Q_0$ the condition $\mu_{Z'}(I)\ne\mu_{Z'}(J)$ removes a hypersurface from the space of all stability conditions. We define $Z\df=Z+\eps Z'$ for $0<\eps\ll1$. One can see that if vectors $u,u',v,v'\in\bC$ are such that $u+\eps u'\parallel v+\eps v'$ for $0<\eps\ll1$ then $u\parallel v$ and $u'\parallel v'$. This implies that $\mu_{Z\df}(I)\ne\mu_{Z\df}(J)$ for any non-empty subsets $I\ne J$ of $Q_0$. The second condition is automatically satisfied.

Now the third claim immediately follows from the first claim: if $M$ is a $Z\df$-semistable abelian representation and $I\sb Q_0$ is a non-empty proper $M$-closed subset, then $\mu_{Z\df}(I)\le\mu_{Z\df}(Q_0)$. By the first claim, we have $\mu_{Z\df}(I)\ne\mu_{Z\df}(Q_0)$, thus $\mu_{Z\df}(I)<\mu_{Z\df}(Q_0)$, proving stability. For the fourth claim, assume $M$ to be $Z\df$-stable, and let $I\sb Q_0$ be a non-empty proper $M$-closed subset. Then $\mu_{Z\df}(I)<\mu_{Z\df}(Q_0)$, thus $\mu_{Z}(I)\le\mu_{Z}(Q_0)$ by the second claim, proving $Z$-semistability of $M$. For the fifth claim, assume $M$ to be $Z$-stable, and let $I$ be a non-empty proper $M$-closed subset. Then $\mu_{Z}(I)<\mu_{Z}(Q_0)$, thus $\mu_{Z\df}(I)<\mu_{Z\df}(Q_0)$ by the second claim. The lemma is proved.
\end{proof}

\begin{corollary}
Let an abelian representation $M\in R(Q)$ have the Harder-Narasimhan type $(I_1,\ldots,I_s)$ (compare the proof of Theorem \ref{th:wall-cross}) with respect to $Z\df$.
Then $M$ is $Z$-semistable if and only if $\mu_Z(I_k)=\mu_Z(Q_0)$ for all $k$.
\end{corollary}
\begin{proof}
Define $I_{\leq k}=I_1\cup\ldots\cup I_k$ for $k=1,\ldots,s$.
Assume that $M$ is $Z$-semistable.
Then the defining condition $$\mu_{Z\df}(I_1)>\mu_{Z\df}(I_2)>\ldots >\mu_{Z\df}(I_s)$$
of a Harder-Narasimhan type implies that $\mu_{Z\df}(I_{\leq k})>\mu_{Z\df}(Q_0)$ for all $k<s$. Since every $I_{\leq k}$ is $M$-closed by definition of the Harder-Narasimhan filtration, we have $\mu_Z(I_{\leq k})\leq \mu_Z(Q_0)$ by $Z$-semistability of $M$. If $\mu_Z(I_{\leq k})<\mu_Z(Q_0)$, then $\mu_{Z\df}(I_{\leq k})<\mu_{Z\df}(Q_0)$ by the second claim of the previous lemma, a contradiction. Thus $\mu_Z(I_{\leq k})=\mu_Z(Q_0)$ for all $k=1,\ldots, s$, which implies $\mu_Z(I_k)=\mu_Z(Q_0)$ for all $k$, proving the first claim.

Now let us prove the converse. By definition of the Harder-Narasimhan type, $M$ admits a filtration by subrepresentations $0=M_0\subset M_1\subset\ldots\subset M_s=M$ such that $M_k/M_{k-1}$ is $Z\df$-semistable with support in $I_k$. By the fourth claim of the previous lemma, each $M_k/M_{k-1}$ is $Z$-semistable, and it is supported on $I_k$ which has the same $Z$-slope as $Q_0$. But an iterated extensions of $Z$-semistable representations of the same $Z$-slope is again $Z$-semistable, proving the converse.
\end{proof}

\begin{proposition}
\label{prp:hn deform}
Let $Q$ be a quiver with a stability $Z$ and let $\al\in\bN^{Q_0}$ be as in Definition \ref{df:ab1}. Let $Z\df$ be a deformation of $Z$ for the quiver $Q(\al)$ as in Lemma \ref{lm:deform}. Then we have an identity of motives
$$[R_Z(Q(\al))]=\sum_{Q(\al)_0=I_1\dot\cup\ldots\dot\cup I_s}q^{\sum_{k<l}r_Q(\pi(I_k),\pi(I_l))}\prod_{k=1}^s[R_{Z\df}(Q(\al)|_{I_k})],$$
where the sum ranges over all unordered disjoint decompositions of $Q(\al)_0$ into subsets $I_k$ such that $\mu_Z(I_k)=\mu_Z(\al)$ for all $k$. 
\end{proposition}
\begin{proof}
By the previous corollary, $[R_Z(Q(\al))]$ equals the sum of the motives of strata having Harder-Narasimhan type $(I_1,\ldots,I_s)$ with respect to $Z\df$ such that $\mu_Z(I_k)=\mu_Z(\al)$ for all $k=1,\ldots,s$. 
Every unordered disjoint decomposition $Q(\al)_0=I_1\dot\cup\ldots\dot\cup I_s$ into subsets such that $\mu_Z(I_k)=\mu_Z(\al)$ for all $k$ corresponds to a unique such Harder-Narasimhan stratum: indeed, by the first claim of the above lemma, all values $\mu_{Z\df}(I_k)$ are pairwise distinct, thus there exists a unique ordering of the $I_k$ with decreasing $Z\df$-slopes.

Every Harder-Narasimhan stratum is isomorphic to the product of
\linebreak
$\prod_k R_{Z\df}(Q(\al)|_{I_k})$ and a certain affine space encoding the scalars representing the arrows between distinct subsets $I_k$, $I_l$. By the assumption of Definition \ref{df:ab1}, the dimension of this affine space is $\sum_{k<l}r_Q(\pi(I_k),\pi(I_l))$, {\it regardless of the ordering of the subsets}. The proposition is proved.
\end{proof}

\begin{corollary}
Under the assumptions of Proposition \ref{prp:hn deform} we have an identity of virtual motives
$$[R_Z(Q(\al))]_\vir=\sum_{Q(\al)_0=I_1\dot\cup\dots\dot\cup I_s}\prod_{k=1}^s[R_{Z\df}(Q(\al)|_{I_k})]_\vir.$$
where the sum ranges over all unordered disjoint decompositions of $Q(\al)_0$ into subsets $I_k$ such that $\mu_Z(I_k)=\mu_Z(\al)$ for all $k$.
\end{corollary}
\begin{proof}
Let $\al^k=\pi(I_k)$ for $k=1,\dots,s$.
Using the symmetry assumption of Definition \ref{df:ab1} we obtain
$$q^{\oh(\n{Q(\al)_1}-\sum_k\n{Q(\al^k)_1})}=q^{\oh\sum_{k\ne l}r_Q(\al^k,\al^l)}=q^{\sum_{k<l}r_Q( \al^k,\al^l)}.$$
\end{proof}

\begin{theorem}
Under the assumptions of Proposition \ref{prp:hn deform} we have
$$g_Z(\al)
=(\gme)\frac{[R_{Z\df}(Q(\al))]_\vir}{[G(Q(\al))]_\vir}
=[M_{Z\df}(Q(\al))]_\vir.$$
\end{theorem}
\begin{proof}
We prove the first equality by induction on $\n\al$. 
For any non-empty proper subset $I\sb Q(\al)_0$ with $\mu_Z(I)=\mu_Z(\al)$, the restriction of $Z\df$ to $Q(\al)|_I$
satisfies the conditions of Lemma \ref{lm:deform}. Therefore we can assume by induction that
$$\frac{[R_{Z\df}(Q(\al)|_I)]_\vir}{[G(Q(\al)|_I)]_\vir}=\frac{g_Z(\pi(I))}{\gme}.$$
Now by the previous corollary, we have
$$\frac{[R_Z(Q(\al))]_\vir}{[G(Q(\al))]_\vir}
=\frac{[R_{Z\df}(Q(\al))]_\vir}{[G(Q(\al))]_\vir}
+\sum_{
\substack{Q(\al)_0=I_1\dot\cup\dots\dot\cup I_s\\s\ge2}}\prod_{k=1}^s\frac{g_Z(\pi(I_k))}{\gme},$$
where the sum ranges over all unordered disjoint decompositions of $Q(\al)_0$ into subsets $I_k$ with $\pi(I_k)\in Z\inv(l)$.
Given an ordered decomposition $\al=\al^1+\ldots+\al^s$ with $\al^k\in Z\inv(l)$, the number of disjoint decompositions $Q(\al)_0=I_1\dot\cup\ldots\dot\cup I_s$ such that $\pi(I_k)=\al^k$ for all $k$ equals the product of multinomial coefficients $\binom\al{\al^1,\ldots,\al^s}$. Therefore
$$\frac{[R_Z(Q(\al))]_\vir}{[G(Q(\al))]_\vir}
=\frac{[R_{Z\df}(Q(\al))]_\vir}{[G(Q(\al))]_\vir}
+\sum_{\substack{\al=\al^1+\dots+\al^s\\s\ge2}}
\frac1{s!}\binom\al{\al^1,\ldots,\al^s}\prod_{k=1}^s\frac{g_Z(\al^k)}{\gme},$$
where the factor $\frac{1}{s!}$ compensates for the fact that the ordering in the disjoint decompositions was disregarded.

Comparing coefficients of $x^\al$ in the defining equation
$$1+\sum_{\al\in Z^{-1}(l)}\frac{[R_Z(Q(\al))]_\vir}{[G(Q(\al))]_\vir}\frac{x^\al}{\al!}
=\exp\bigg(\frac{1}{\gme}\sum_{\al\in Z^{-1}(l)}g_Z(\al)\frac{x^\al}{\al!}\bigg)$$
of the abelian invariants $g_Z(\al)$, we see that
$$\frac{[R_Z(Q(\al))]_\vir}{[G(Q(\al))]_\vir}=\sum_{\al=\al^1+\ldots+\al^s}\frac{1}{s!}
\binom\al{\al^1,\dots,\al^s}
\prod_{k=1}^s\frac{g_Z(\al^k)}\gme.$$
This implies
$$\frac{[R_{Z\df}(Q(\al))]_\vir}{[G(Q(\al))]_\vir}
=\frac{g_Z(\al)}{\gme}.$$

The second equality of the theorem follows from the fact that the principal $G(Q(\al))/\Gm$-fibration $R_{Z\df}(Q(\al))\to M_{Z\df}(Q(\al))$ is Zariski-locally trivial by Hilbert's Theorem 90.
\end{proof}

\begin{corollary}
\label{cr:positivity}
 Under the assumptions of Definition \ref{df:ab1}, we have $g_Z(\al)\in\bN[q^{\pm\oh}]$.
\end{corollary}
\begin{proof}
By the motivic nature of the Harder-Narasimhan recursion, the motive of $M_{Z\df}(Q(\al))$ is a polynomial in $q$, which actually equals the count over finite fields. It follows from Proposition \ref{purity} that this polynomial is in $\bN[q]$. This fact, together with the previous theorem, implies that $g_Z(\al)\in\bN[q^{\pm\oh}]$.
\end{proof}

\begin{remark}
We proved in the above theorem that for any $I,J\sb Q(\al)_0$ with $\pi(I)=\pi(J)$, the moduli spaces $M_{Z\df}(Q(\al)|_I)$ and $M_{Z\df}(Q(\al)|_J)$ have equal motives. We don't claim however that these moduli spaces are isomorphic to each other. Moreover, for different choices of $Z\df$, we get deformed stabilities in many different chambers of stability space and thus possibly many nonisomorphic moduli spaces, but nevertheless the theorem proves that their motives are the same.
\end{remark}

\begin{remark}
The map $M_{Z\df}(Q(\al))\rightarrow M_Z(Q(\al))$ (well-defined by the first lemma of this section) is a desingularization of $M_Z(Q(\al))$. It would be interesting to study its geometry in more detail.
\end{remark}

\begin{remark} Combining the theorem in this section with the main result of section 4, we get an explicit formula for the motive of $M_{\mathrm{triv}\df}(Q(\al))$ in terms of spanning trees.
\end{remark}

\section{Indecomposable abelian representations}\label{sec:indec}
Let $Q$ be a quiver and let $M$ be an abelian representation of $Q$. Define the support quiver $Q_M$ of $M$ as the subquiver of $Q$ with the same set of vertices, and arrows $a\in Q_1$ whenever $M_a\ne0$ (that is, the linear map between one-dimensional vector spaces representing the arrow is non-zero). In general, we call a subquiver $G\sb Q$ spanning if $G_0=Q_0$. By associating to an abelian representation $M$ its support quiver and the scalars representing the non-zero arrows, we get immediately:

\begin{lemma}\label{mrl1}
There is a bijection between points in $R(Q)$ and tuples consisting of a spanning subquiver $G\sb Q$ and a choice of a non-zero scalar in ${\bk}$ for every arrow in $G$.
\end{lemma}

Also the proof of the following lemma is immediate:

\begin{lemma}\label{mrl2}
$M$ is indecomposable if and only if $Q_M$ is connected. In this case, $M$ is absolutely indecomposable and Schurian, that is, $\End(M)=\bk$.
\end{lemma}

\begin{definition}
Let $Q$ be a quiver and $Z:\Ga(Q)\to\bC$ be a stability function.
\begin{enumerate}
\item As before, for any subset $I\sb Q_0$, we define $\one_I=\sum_{i\in I}i\in \bZ^{Q_0}$ and $\mu_Z(I)=\mu_Z(\one_I)$. 
\item A subset $I\sb Q_0$ is called $Q$-closed if there are no arrows in $Q$ from $I$ to $Q_0\ms I$.
\item The quiver $Q$ is called $Z$-semistable (resp.~$Z$-stable) if for any proper $Q$-closed subset $I\sb Q_0$ we have $\mu_Z(I)\le\mu_Z(Q_0)$ (resp.~$\mu_Z(I)<\mu_Z(Q_0)$).
\item The set of all spanning $Z$-semistable subquivers $G\sb Q$ is denoted by $\cG_Z(Q)$. The subset of connected $G\in\cG_Z(Q)$ is denoted by $\cC_Z(Q)$.
\end{enumerate}
\end{definition}

Using this terminology, the following is immediate:

\begin{lemma}\label{mrl3}
An abelian representation $M$ of $Q$ is (semi)stable if and only if its support quiver $Q_M$ is (semi)stable.
\end{lemma}

\begin{remark}
A stability $Z$ is $\one$-generic (\ie every representation in $R_Z(Q)$ is stable) if and only if $\cG_Z(Q)=\cC_Z(Q)$.
\end{remark}

The above lemma yields a bijection between the points of the semistable locus $R_Z(Q)$ in $R(Q)$ and tuples consisting of a subquiver $G\in\cG_Z(Q)$ together with a choice of a non-zero scalar in $\bk$ for every arrow in $G$. Similarly, there is a bijection between the points in the indecomposable semistable locus $R^\ind_Z(Q)$ in $R(Q)$ and tuples consisting of a subquiver $G\in\cC_Z(Q)$ together with a choice of a non-zero scalar in $\bk$ for every arrow in $G$.
This allows us to calculate
\begin{equation}
\frac{\nnn{R_Z(Q)}}{\nnn{G(Q)}}
=\sum_{G\in\cG_Z(Q)}(q-1)^{\n{G_1}-\n{G_0}}
\label{eq:st1}
\end{equation}
and
\begin{equation}
\frac{\nnn{R^\ind_Z(Q)}}{\nnn{G(Q)}}
=\sum_{G\in\cC_Z(Q)}(q-1)^{\n{G_1}-\n{G_0}}.
\label{eq:st2}
\end{equation}
We define $a_Z(Q):=(q-1)\frac{\nnn{R^\ind_Z(Q)}}{\nnn{G(Q)}}$.

\begin{remark}
We can interpret $a_Z(Q)$ as a polynomial in $q$. Given a field $\bk$, let $M_Z^\ind(Q)(\bk)$ be the set of isomorphism classes of abelian indecomposable $Z$-semistable $Q$-representations over $\bk$. Then for any finite field $\bF_q$ we have $a_Z(Q)(q)=\n{M_Z^\ind(Q)(\bF_q)}$.
\end{remark}

\begin{definition}
For any quiver (or graph) $G$ we define $\ka(G)$ to be the number of connected components of $G$ and we define the nullity of $G$ to be $n(G)=\n{G_1}-\n{G_0}+\ka(G)$. We always have $n(G)\ge0$ and $n(G)=0$ if and only if $G$ is a forest.
\end{definition}

Using this notation we can write
\begin{equation}
a_Z(Q)=\sum_{G\in\cC_Z(Q)}(q-1)^{n(G)}.
\end{equation}
Applying the exponential formula as in \cite[Corollary 5.1.6]{stanley_enumerative2} and using formulas \eqref{eq:st1} and \eqref{eq:st2}, we get:

\begin{theorem}
For any ray $l\in\bH_+$, we have
\begin{equation}
1+\sum_{Z(\al)\in l}\frac{\nnn{R_Z(Q(\al))}}{\nnn{G(Q(\al))}}\frac{x^\al}{\al!}
=\exp\bigg(\sum_{Z(\al)\in l}\frac{\nnn{R_Z^\ind(Q(\al))}}{\nnn{G(Q(\al))}}\frac{x^\al}{\al!}\bigg).
\label{eq:exp}
\end{equation}
\end{theorem}

\begin{corollary}
Assume that a ray $l\sb\bH_+$ satisfies the assumptions of Definition \ref{df:ab1}. 
Then we have
\begin{equation*}
T_{r}\exp\bigg(\sum_{Z(\al)\in l}\frac{g_Z(\al)}{q^\oh-q^{-\oh}}\frac{x^\al}{\al!}\bigg)
=\exp\bigg(\sum_{Z(\al)\in l}q^{\oh\n\al}\frac{a_Z(Q(\al))}{q-1}\frac{x^\al}{\al!}\bigg).
\end{equation*}
In particular, $a_Z(Q(\al))\in\bN[q]$ by Corollary \ref{cr:positivity} and Theorem \ref{th:2}.
\end{corollary}
\begin{proof}
By the defining equation of abelian quiver invariants, we have
\begin{multline*}
T_{r}\exp\bigg(\sum_{Z(\al)\in l}\frac{g_Z(\al)}{q^\oh-q^{-\oh}}\frac{x^\al}{\al!}\bigg)
=T_r\bigg(1+\sum_{Z(\al)\in l} \frac{\nnn{R_Z(Q(\al))}_\vir}{\nnn{G(Q(\al))}_\vir}
\frac{x^\al}{\al!}\bigg)\\
=1+\sum_{Z(\al)\in l} q^{\oh\n\al}\frac{\nnn{R_Z(Q(\al))}}{\nnn{G(Q(\al))}}
\frac{x^\al}{\al!}.
\end{multline*}
This expression is equal, by the previous theorem, to
$$\exp\bigg(\sum_{Z(\al)\in l}q^{\oh\n\al}
\frac{\nnn{R_Z^\ind(Q(\al))}}{\nnn{G(Q(\al))}}\frac{x^\al}{\al!}\bigg)
=\exp\bigg(\sum_{Z(\al)\in l}q^{\oh\n\al}
\frac{\nnn{M_Z^\ind(Q(\al))}}{q-1}\frac{x^\al}{\al!}\bigg)
$$
\end{proof}

\begin{remark}
If $Q$ is a symmetric quiver then for the trivial stability the value of $g_\triv(\al)$ at $q^\oh=1$ equals the number of spanning trees of $\ub Q(\al)$ (see Corollary \ref{triv_special}). On the other hand the value of $a_\triv(Q(\al))$ at $q=1$ equals the number of spanning trees of $Q(\al)$.
\end{remark}

We will see later that $a_Z(Q)\in\bN[q]$ without the assumptions of Definition \ref{df:ab1}. 
If $Z$ is a trivial stability, then there are two classical interpretations of the polynomial $a_Z(Q)$.

\begin{remark}
For a trivial stability $Z$ the set $\cC_Z(Q)$ coincides with the set $\cC(Q)$ of all connected spanning subgraphs of $Q$. For any graph $Q$ one defines its Tutte polynomial \cite{bollobas_modern} by
$$T(Q;t,q)=\sum_{G\in\cG(Q)}(t-1)^{\ka(G)-\ka(Q)}(q-1)^{n(G)}.$$
It is known that $T(Q;t,q)\in\bN[t,q]$. If $Q$ is connected, then
$$T(Q;1,q)=\sum_{G\in\cC(Q)}(q-1)^{n(G)}=a_Z(Q),$$
so $a_Z(Q)\in\bN[q]$.
\end{remark}

\begin{remark}
For a trivial stability $Z$ the set $M^\ind_Z(Q)(\bk)$ coincides with the set $M^\ind(Q,\one)(\bk)$ of isomorphism classes of (absolutely) indecomposable $Q$-represen\-tations over \bk having dimension vector \one. Therefore $a_Z(Q)\in\bN[q]$ by the Kac conjecture proved by Crawley-Boevey and Van den Bergh \cite{crawley-boevey_absolutely} in the case of indivisible dimension vectors (in particular, for the dimension vector $\one\in\bZ^{Q_0}$) and by Hausel, Letellier, and Rodriguez-Villegas \cite{hausel_positivity} in general.
\end{remark}

\begin{lemma}
If $Q$ is connected then there exists a stability function $Z:\Ga(Q)\to\bC$ such that $Q$ is $Z$-stable.
\end{lemma}
\begin{proof}
Deleting arrows if necessary, we can assume that $Q$ is a tree.
Then we can find a vertex $i_0\in Q_0$ that is incident to just one arrow. Without loss of generality we can assume that $i_0$ is the start point of this arrow. The quiver $Q'=Q\ms\set{i_0}$ is connected. By induction on $\n{Q_0}$ we can assume that there exists $d':Q'_0\to\bR$ such that $d'(Q'_0)=0$ and $d'(I)<0$ for any proper $Q'$-closed subset $I\sb Q'_0$ (we define $d'(I)=\sum_{i\in I}d'(i)$). Let $\eps>0$ be the minimum of $\n{d'(I)}$ over all such subsets.
Define $d(i)=d'(i)-\frac\eps{\n{Q'_0}}$ for $i\ne i_0$ and $d(i_0)=\eps$. Then $d(Q_0)=d'(Q'_0)=0$. For any proper $Q$-closed subset $I\sb Q_0$, if $i_0\in I$ and $I'=I\ms\set {i_0}$, then
$$d(I)=d'(I')-\frac{\n{I'}}{\n{Q'_0}}\eps+\eps<d'(I')+\eps\le 0.$$
If $i_0\notin I$ then
$$d(I)=d'(I)-\frac{\n{I}}{\n{Q'_0}}\eps<0.$$
This implies that $Q$ is $Z$-stable with respect to $Z=-d+\sqrt{-1}r$, where $r(i)=1$ for all $i\in Q_0$.
\end{proof}

\begin{proposition}
\label{pr:tree}
Let $Q$ be a quiver with a stability function $Z$ such that $Q$ is the only quiver in $\cC_Z(Q)$. Then $Q$ is a tree. 
\end{proposition}
\begin{proof}
By the previous lemma there exists a stability function $Z'$ such that $Q$ is $Z'$-stable. Let $Z''=Z+\eps Z'$ for $0<\eps\ll1$. Then $Q$ is $Z''$-stable and is the only quiver in $\cC_{Z''}(Q)$. Deforming $Z''$ as in Lemma \ref{lm:deform}, we can assume that $Z''$ is moreover \one-generic.
By Proposition \ref{purity} the counting polynomial of $M_{Z''}(Q)$ has non-negative coefficients. If $M\in R_{Z''}(Q)$ then $M$ is $Z''$-stable and in particular indecomposable. This implies that $Q_M\in\cC_{Z''}(Q)$ and $Q_M=Q$. Therefore the counting polynomial of $M_{Z''}(Q)$ equals $(q-1)^{n(Q)}$. Therefore $n(Q)=0$ and $Q$ is a tree.
\end{proof}

\begin{theorem}
\label{th:pos1}
Let $Q$ be a quiver with a stability function $Z$. Then $a_Z(Q)\in\bN[q]$.
\end{theorem}
\begin{proof}
Given a quiver $G\in\cC_Z(Q)$, let
$$A(G)=\sets{a\in G_1}{G\ms\set a\in\cC_Z(Q)}.$$
We denote the set of graphs $G\in\cC_Z(Q)$ with $A(G)=\es$ by $\cT_Z(Q)$. By Proposition \ref{pr:tree} the family $\cT_Z(Q)$ consists of trees.

We choose a total order on $Q_1$. For any $G\in\cC_Z(G)\ms\cT_Z(G)$, let $a(G)=\min A(G)\in G_1$. Deleting the arrow $a(G)$ from $G$ and continuing this process we will eventually obtain some $T\in\cT_Z(Q)$. Conversely, given $T\in\cT_Z(Q)$ let
$$E(T)=\sets{b\in Q_1\ms T_1}{b=a(T\cup\set b)}.$$
Then for any subset $J\sb E(T)$ the quiver $G=T\cup J$ is contained in $\cC_Z(Q)$  and $T$ is obtained from $G$ by the above process. This implies that
$$a_Z(Q)=\sum_{G\in\cC_Z(Q)}(q-1)^{n(G)}
=\sum_{T\in\cT_Z(Q)}(q-1)^{n(T)}q^{\n{E(T)}}
=\sum_{T\in\cT_Z(Q)}q^{\n{E(T)}}\in\bN[q].$$
\end{proof}

In view of the above theorem we can formulate the following generalization of the Kac conjecture.

\begin{conjecture}
Let $Q$ be a quiver with a stability function $Z$ and let $\al\in\Ga_+(Q)$ be a dimension vector.
Given a field $\bk$, let $M^\ind_Z(Q,\al)(\bk)$ be the set of isomorphism classes of absolutely indecomposable $Z$-semistable $Q$-representations over $\bk$ having dimension vector \al. Then there exists a polynomial $a_Z(Q,\al)\in\bN[q]$ such that for any finite field $\bF_q$ we have
$a_Z(Q,\al)(q)=\n{M^\ind_Z(Q,\al)(\bF_q)}$.
\end{conjecture}

We can generalize the Tutte polynomial for the case of a quiver with a stability function.

\begin{conjecture}[Semistable Tutte polynomial]
Let $Q$ be a quiver with a stability function $Z$.
Define the semistable Tutte polynomial by 
$$T_Z(Q;t,q)=\sum_{G\in\cG_Z(Q)}(t-1)^{\ka(G)-\ka(Q)}(q-1)^{n(G)}.$$
We conjecture that $T_Z(Q;t,q)\in\bN[t,q]$. The proof should not be very different from Theorem \ref{th:pos1}.
\end{conjecture}

\begin{remark}
If $Z$ is \one-generic then $\cG_Z(Q)=\cC_Z(Q)$. This implies
$$T_Z(Q;t,q)=T_Z(Q;1,q)=\sum_{G\in\cC_Z(Q)}(q-1)^{n(G)}=a_Z(Q)\in\bN[q].$$
\end{remark}

\section{MPS wall-crossing formula}\label{mpswcf}
Let \Ga be a rank $2$ lattice with a skew-symmetric form $\ang{\cdot\,,\cdot}:\Ga\xx\Ga\to\bZ$ and a basis $(e_1,e_2)$ such that $\ang{e_1,e_2}>0$.
Let $\Ga_+=\bN e_1+\bN e_2\iso\bN^2$ and $\Ga_+^*=\Ga_+\ms\set0$.
We consider stability functions $Z:\Ga\to\bC$ such that
$$Z(e_i)=-d_i+\sqrt{-1}r_i\in\bH_+,\qquad i=1,2.$$
For any $\al=(\al_1,\al_2)\in\Ga_+^*$, we define its slope $\mu_Z(\al)\in\bR\cup\set\infty$ by
$$\mu_Z(\al)=\frac{d_1\al_1+d_2\al_2}{r_1\al_1+r_2\al_2}$$
and define a total preorder $\preceq_Z$ on $\Ga_+^*$ by the rule $\al\preceq_Z\be$ if $\mu_Z(\al)\le\mu_Z(\be)$.

Although the set $\bH_+^2$ of different stability functions $Z:\Ga\to\bC$ is huge, these functions can induce just three different total preorders depending on the inequality between $\mu_Z(e_1)$ and $\mu_Z(e_2)$. 
We will say that a stability function $Z$ is trivial (belongs to the marginal wall) if $\mu_Z(e_1)=\mu_Z(e_2)$.
We will denote by $c_+$ the set (chamber) of stability functions $Z$ such that $\mu_Z(e_1)<\mu_Z(e_2)$ and we will denote by $c_-$ the set (chamber) of stability functions $Z$ such that $\mu_Z(e_1)>\mu_Z(e_2)$.

\begin{remark}
Let $Q$ be the generalized Kronecker quiver with two vertices $1,2$ and with $m>0$ arrows from $2$ to $1$. Then $\Ga=\Ga(Q)\iso\bZ^2$ and the skew-symmetric form $\ang{\al,\be}=\hi(\al,\be)-\hi(\be,\al)$ satisfies $\ang{e_1,e_2}=m>0$. In the chamber $c_+$ we have $e_1\prec e_2$ and there exist plenty of stable representations of $Q$.
On the other hand, in the chamber $c_-$ we have $e_2\prec e_1$ and the only stable representations of $Q$ are one-dimensional.
\end{remark}

\begin{lemma}
Let $Z:\Ga\to\bC$ be some stability function.
\begin{enumerate}
	\item If $Z\in c_+$ then $\al\preceq_Z\be$ if and only if $\ang{\al,\be}\ge0$.
	\item If $Z\in c_-$ then $\al\preceq_Z\be$ if and only if $\ang{\al,\be}\le0$.
\end{enumerate}
\end{lemma}
\begin{proof}
Let $Z\in c_+$ and let $\mu_Z(e_1)=\frac{d_1}{r_1}$, $\mu_Z(e_2)=\frac{d_2}{r_2}$.
Then $\mu_Z(e_1)<\mu_Z(e_2)$ and therefore $d_1r_2-d_2r_1<0$.
We have $\al\preceq_Z\be$ if and only if
$$\frac{d_1\al_1+d_2\al_2}{r_1\al_1+r_2\al_2}\le
\frac{d_1\be_1+d_2\be_2}{r_1\be_1+r_2\be_2},$$
that is,
$$(\al_1\be_2-\al_2\be_1)(d_1r_2-d_2r_1)\le0$$
and
$$\al_1\be_2-\al_2\be_1\ge0.$$
This is equivalent to
$$\ang{\al,\be}=(\al_1\be_2-\al_2\be_1)\ang{e_1,e_2}\ge0.$$
\end{proof}

We will use the slope ordering of the vectors (and rays) in the first quadrant of $\Ga_\bR=\Ga\ts\bR\iso\bR^2$ (for example $e_1<e_2$). This corresponds to the ordering with respect to stability functions in the chamber $c_+$. Assume that for any $\ga\in\Ga_+^*$ we have invariants $\ob_\ga,\oa_\ga\in\cV$ (rational DT invariants in the chambers $c_+$ and $c_-$, respectively) related by the equation (KS wall-crossing formula or HN recursion)
\begin{equation}
\prod^{\curvearrowright}_l\exp\left(\frac{\sum_{\ga\in l\cap\Ga}\ob_\ga x^\ga}{\gme}\right)
=\prod^{\curvearrowleft}_l\exp\left(\frac{\sum_{\ga\in l\cap\Ga}\oa_\ga x^\ga}{\gme}\right)
\label{eq:ks1again}
\end{equation}
in the quantum torus of $(\Ga,\ang{\cdot\,,\cdot})$. Then we can express
\begin{equation}
\ob_\ga=\sum_{\over{m:\Ga_+^*\to\bN}{\nn m=\ga}}\frac{g(m)}{m!}\prod_{\al\in\Ga_+^*}(\oa_\al)^{m(\al)}
\label{eq:mps1again}
\end{equation}
for some invariants $g(m)$. Our goal is to determine these invariants.

Let $\hq$ be the quiver with set of vertices $\Ga_+^*$ and with the number of arrows from $\al$ to $\be$ equal to $\ang{\be,\al}$ if $\ang{\be,\al}>0$, and zero otherwise. 
Note that $\Ga_+(\hq)=\cP(\Ga_+^*)$, \ie a map $m:\Ga_+^*\to\bN$ with finite support can be identified with a dimension vector of $\hq$.
The natural group homomorphism
$$\nn-:\Ga(\hq)\to\Ga,\qquad m\mto\sum_{\al\in\Ga_+^*}m(\al)\al.$$
preserves the skew-symmetric forms and therefore induces an algebra homomorphism of the corresponding quantum tori. Any stability function on $\Ga$ induces a stability function on $\Ga(\hq)$ and therefore a total preorder on $\Ga_+^*(Q)$.
As before, we consider only stability functions $Z$ on \Ga from the chamber $c_+$. Then, for $m,m'\in\Ga_+^*(Q)$, we have $\mu_Z(m)\le \mu_Z(m')$ if and only if $\ang{m,m'}\ge0$. This implies that $\mu_Z(m)=\mu_Z(m')$ if and only if $\ang{m,m'}=0$, and we can define abelian quiver invariants $g(m)=g_+(m)$ for any $m\in\cP(\Ga_+^*)=\Ga_+(\hq)$ by the formula
\begin{equation}
1+\sum_{\nn m\in l}f_+(m)\frac{x^m}{m!}=\exp\left(\sum_{\nn m\in l}\frac{g(m)}{\gme}\frac{x^m}{m!}\right),
\label{eq:mps sol2}
\end{equation}
where $f_+(m)$ is the  motivic invariant of the moduli stack of the abelian semistable representations of the quiver $\hq(m)$ (see \ref{sec:abelian}).

\begin{theorem}
Assume that the invariants $g(m)$ for $m\in\cP(\Ga_+^*)=\Ga_+(\hq)$ are the abelian quiver invariants of the quiver $Q$. Then the invariants $\ob(\ga),\oa(\ga)$ for $\ga\in\Ga_+^*$ satisfy the KS wall-crossing formula \eqref{eq:ks1again} if and only if they satisfy \eqref{eq:mps1again}.
\end{theorem}
\begin{proof}
We will show that the formula \eqref{eq:mps1again} with the $g(m)$ being abelian quiver invariants implies the KS wall-crossing formula \eqref{eq:ks1again}. Using formula \eqref{eq:mps1again} we obtain
$$\sum_{\ga\in l}\ob_\ga x^\ga
=\sum_{\over{m\in\cP(\Ga_+^*)}{\nn m\in l}}\frac{g(m)}{m!}\bigg(\prod_{\ga\in\Ga_+^*}(\oa_\ga)^{m(\ga)}\bigg)x^{\nn m}
=\sum_{\over{m\in\cP(\Ga_+^*)}{\nn m\in l}}g(m)
\prod_{\ga\in\Ga_+^*}\frac{(\oa_\ga x^\ga)^{m(\ga)}}{m(\ga)!}
.$$
Then the left hand side of \eqref{eq:ks1again} can be written as
\begin{equation*}
\prod^{\curvearrowright}_l\exp
\Bigg(\sum_{\over{m\in\cP(\Ga_+^*)}{\nn m\in l}}
\frac{g(m)}{\gm}\prod_{\ga\in\Ga_+^*}\frac{(\oa_\ga x^\ga)^{m(\ga)}}{m(\ga)!}\Bigg).
\label{eq:left1}
\end{equation*}
For any vertex $\ga\in\hq_0=\Ga_+^*$ there is a variable $x_\ga$ in the quantum torus of $Q$. We substitute each element $\oa_\ga x^\ga$ in the quantum torus of $(\Ga,\ang{\cdot\,,\cdot})$ by the element $x_{\ga}$ in the quantum torus of $Q$. Then \eqref{eq:ks1again} can be written as
$$
\prod^{\curvearrowright}_l
\exp
\Bigg(\sum_{\over{m\in\cP(\Ga_+^*)}{\nn m\in l}}
\frac{g(m)}{\gm}\prod_{\ga\in\Ga_+^*}\frac{x_\ga^{m(\ga)}}{m(\ga)!}\Bigg)
=\prod_{l}^\curvearrowleft\exp\Bigg(
\sum_{\ga\in l}
\frac{x_{\ga}}\gm\Bigg).$$
We note that
$$
\exp\Bigg(\sum_{\nn m\in l}
\frac{g(m)}{\gm}\prod_{\ga\in\Ga_+^*}\frac{x_\ga^{m(\ga)}}{m(\ga)!}\Bigg)
=
\exp\Bigg(\sum_{\nn m\in l}
\frac{g(m)}{\gm}\frac {x^m}{m!}\Bigg).
$$
To prove formula \eqref{eq:ks1again}, we have to show that
$$\prod_l^\curvearrowright\bigg(1+\sum_{\nn m\in l}f_+(m)\frac{x^m}{m!}\bigg)
=\prod_{l}^\curvearrowleft
\exp\Bigg(\sum_{\ga\in l}\frac{x_{\ga}}\gm\Bigg).$$
This will follow from Theorem \ref{th:wall-cross} once we show that the right hand side corresponds to the $c_-$-semistable abelian representations of $\hq$.

Let $m\in\Ga_+(\hq)$ be such that there exist $c_-$-semistable abelian representations of $\hq(m)$. Assume that there are vertices $\ga_i$, $\ga'_j$ of $\hq(m)$ such that $\ga_i<\ga'_j$ (in $c_-$). This means that $\ang{\ga,\ga'}<0$ and there are $\ang{\ga',\ga}$ arrows from $\ga_i$ to $\ga'_j$. Let $\ga_i$ be one of the maximal vertices of $\hq(m)$. Then there are no arrows from $\ga_i$ to other vertices of $\hq(m)$. Therefore, for any abelian representation $M$ of $\hq(m)$, there is a subrepresentation $M'$ concentrated at $\ga_i$ and we have $\udim M<\udim M'$ in $c_-$. This implies that there are no $c_-$-semistable abelian representations of $\hq(m)$ unless we have $\ang{\ga,\ga'}=0$ for any $\ga,\ga'\in\supp m$. This means that $\supp m$ is contained in $l=\bR_{>0}\ga_0$ for some $\ga_0\in\Ga_+^*$. The corresponding invariant of $c_-$-semistable abelian representations is $f_-(m)=\gm^{-\n m}$ (there are no arrows in the quiver $\hq(m)$). 
The sum of these invariants (for a fixed ray $l$) is
$$
\sum_{\supp m\sb l}\gm^{-\n m}\frac{x^m}{m!}
=\sum_{\supp m\sb l}\prod_{\ga\in l}\frac{(\gm^{-1}x_{\ga})^{m(\ga)}}{m(\ga)!}
=\prod_{\ga\in l}\exp\left(\gm\inv x_{\ga}\right)
=\exp\Bigg(\sum_{\ga\in l}\frac{x_{\ga}}\gm\Bigg).$$
This finishes the proof of the theorem.
\end{proof}

\subsection{Interpretation of the MPS formula}\label{mpsformula}
We will interpret here the MPS formula \cite[Eq.~D.6]{manschot_wall} as described in \cite[Theorem 3.5]{reineke_mps} using the language of the previous sections.

Let $Q$ be a quiver with a stability function $Z:\Ga(Q)\to\bC$.
Consider the quiver $Q'$ with vertices $i_l$, where $i\in Q_0$, $l\ge1$. The number of arrows from $i_l$ to $j_{l'}$ is defined as $ll'$ times the number of arrows from $i$ to $j$. Consider the projection $\nn-:\Ga(Q')\to\Ga(Q)$, $i_l\mto li$. It induces a stability function $Z:\Ga(Q')\to\bC$ on $Q'$ and it preserves
the skew-symmetric forms on $Q'$ and $Q$. This implies that it induces an algebra morphism from the quantum torus of $Q'$ to the quantum torus of $Q$.
For $\al\in\Ga_+^*(Q)$ let
$$A_Z(Q,\al)=\frac{[R_Z(Q,\al)]_\vir}{[\GL_\al]_\vir}$$
be the motivic invariant of the moduli stack of $Z$-semistable $Q$-representations having dimension vector \al.
For $m\in\Ga_+^*(Q')$ let
$$f_Z(Q'(m))=\frac{[R_Z(Q'(m))]_\vir}{[(\Gm)^{\n m}]_\vir}$$
be the motivic invariant of the moduli stack of $Z$-semistable abelian representations of $Q'(m)$.

We can write \cite[Theorem 3.5]{reineke_mps} in the form
\begin{equation}
A_Z(Q,\al)=\sum_{\nn m=\al}f_Z(Q'(m))\frac{1}{m!}\prod_{i_l\in Q'_0}\left(\frac{\gme}{l(q^{\frac l2}-q^{-\frac l2})}\right)^{m(i_l)}.
\end{equation}
Consider the algebra homomorphism of quantum tori $\pi:\bT_{Q'}\to\bT_{Q}$ given by
$$x^m\mto x^{\nn m}\prod_{i_l\in Q'_0}\left(\frac{\gme}{l(q^{\frac l2}-q^{-\frac l2})}\right)^{m(i_l)}.
$$
Then the previous equation says that this map sends
$$1+\sum_{Z(m)\in l} f_Z(Q'(m))\frac{x^m}{m!}\mto1+\sum_{Z(\al)\in l} A_Z(Q,\al)x^\al$$
for any ray $l\sb\bH_+$.
It is enough to prove this statement only for a trivial stability, as both sides satisfy the HN recursion (for abelian representations this is proved in Theorem 3.3)
and $\pi$ is an algebra homomorphism. For the trivial stability this statement is proved in the first part of \cite[Theorem 3.5]{reineke_mps}.

\providecommand{\bysame}{\leavevmode\hbox to3em{\hrulefill}\thinspace}
\providecommand{\href}[2]{#2}

 
\end{document}